\documentclass[11pt, letter]{amsart}

\usepackage[margin=1.005in, includefoot]{geometry}

\usepackage[colorlinks = true,
            linkcolor = red,
            urlcolor  = magenta,
            citecolor = black,
            anchorcolor = blue]{hyperref}
\usepackage{color}
\usepackage[shortalphabetic,initials,nobysame]{amsrefs}
\usepackage{upgreek}
\usepackage{tikz}
\usepackage[applemac]{inputenc} 
\usetikzlibrary{arrows}
\usepackage{xcolor}
\usepackage[all]{xy}
\usepackage{graphicx}
\usepackage{wrapfig}
\usepackage{yfonts}
\usepackage{graphicx}
\usepackage{amssymb}
\usepackage{epstopdf}
\usepackage{mathrsfs}
\usepackage[mathcal]{eucal}
\usepackage{bm}

\renewcommand{\eprint}[1]{\href{https://arxiv.org/abs/#1}{#1}}

\DeclareMathOperator{\Hom}{Hom}

\BibSpec{article}{%
    +{}  {\PrintAuthors}                {author}
    +{,} { \textit}                     {title}
     +{,} {}                            {note}
    +{.} { }                            {part}
    +{:} { \textit}                     {subtitle}
    +{,} { \PrintContributions}         {contribution}
    +{.} { \PrintPartials}              {partial}
    +{,} { }                            {journal}
    +{}  { \textbf}                     {volume}
    +{}  { \PrintDate}                {date}
    +{,} { \issuetext}                  {number}
    +{,} { \eprintpages}                {pages}
    +{,} { }                            {status}
    +{,} { \DOI}                   {doi}
    +{,} { \eprint}        {eprint}
      +{,} {\publisher}                {publisher}
    +{,} { \address}                   {address}
    +{}  { \parenthesize}               {language}
    +{}  { \PrintTranslation}           {translation}
    +{;} { \PrintReprint}               {reprint}
    +{.} {}                             {transition}
    +{}  {\SentenceSpace \PrintReviews} {review}
}

\newtheorem{thm}{Theorem}

\newcommand{\mf}[1]{\mathfrak{#1}}

\newcommand{\pmat}[4]{\begin{pmatrix}
#1 & #2 \\
#3 & #4
\end{pmatrix}}

\newcommand{\param}[4]{ #1(1+\frac12 #2^2 #3 #4)
            \begin{pmatrix}
            1 & #3 \\
            0 & 1
            \end{pmatrix}
            \begin{pmatrix}
            #2^{-1} & 0 \\
            0 & #2
            \end{pmatrix}
            \begin{pmatrix}
            1 & 0 \\
            #4 & 1 
            \end{pmatrix}}
            
\newcommand{\parami}[4]{ #1^{-1}(1+\frac12 #2^2 #3 #4)
            \begin{pmatrix}
            1 & -#2^2#3 \\
            0 & 1
            \end{pmatrix}
            \begin{pmatrix}
            #2 & 0 \\
            0 & #2^{-1}
            \end{pmatrix}
            \begin{pmatrix}
            1 & 0 \\
            -#2^2#4 & 1 
            \end{pmatrix}}

\newcommand{\rmk}[0]{\paragraph{\textbf{Remark}.}}

\newtheorem{Thm}{Theorem}[section]

\theoremstyle{definition}
\newtheorem{Def}[Thm]{Definition}
\theoremstyle{remark}

\newtheoremstyle{named}{}{}{\itshape}{}{\bfseries}{.}{.5em}{#1 #3}
\theoremstyle{named}

\def\R{\mathbb{R}}

\def\g{\mathfrak{g}}
\def\Frenkel:2013uda{\mathfrak{h}}

\def\a{\alpha}
\def\b{\beta}

\def\bo{\textbf{o}}

\def\=>{\Longrightarrow}

\def\to{\longrightarrow}

\def\o+{\oplus}
\def\bo+{\bigoplus}

\def\<{\langle}
\def\>{\rangle}
\def\({\left(}
\def\){\right)}

\def\^{\wedge}
\def\+{\dagger}

\def\dd[#1,#2]{\frac{d#1}{d#2}}
\def\del[#1,#2]{\frac{\partial #1}{\partial #2}}
\def\over[#1]{\overline{#1}}

\makeatletter
\def\mr@ignsp#1 {\ifx\:#1\@empty\else #1\expandafter\mr@ignsp\fi}%
\newcommand{\multiref}[1]{\begingroup
\xdef\mr@no@sparg{\expandafter\mr@ignsp#1 \: }%
\def\mr@comma{}%
\@for\mr@refs:=\mr@no@sparg\do{\mr@comma\def\mr@comma{,}\ref{\mr@refs}}%
\endgroup}
\makeatother

\newcommand{\hypref}[2]{\ifx\href\asklFrenkel:2013udaas #2\else\href{#1}{#2}\fi}

\tikzset{->-/.style={decoration={
  markings,
  mark=at position .5 with {\arrow{latex}}},postaction={decorate}}}
\tikzset{
    >=latex
    }

\newcommand{\nc}{\newcommand}
\nc{\on}{\operatorname}
\nc{\la}{\lambda}
\nc{\wh}{\widehat}
\nc{\ghat}{\wh\g}
\nc{\mb}{\mathbf}

\begin{document}

\title{Flat $GL(1|1)$-connections and fatgraphs}

\author[A. Bourque]{Andrea Bourque}

\author[A.M. Zeitlin]{Anton M. Zeitlin}
\address{
          Department of Mathematics, 
          Louisiana State University, 
          Baton Rouge, LA 70803, USA}

\date{\today}

\begin{abstract}
We study the moduli space of flat $GL(1|1)$-connections on a punctured surface from the point of view of graph connections. To each fatgraph, a system of coordinates is assigned, which involves two bosonic and two fermionic variables per edge, subject to certain relations. In the case of trivalent graphs, we provide a closed explicit formula for the Whitehead moves. In addition, we discuss the invariant Poisson bracket.
\end{abstract}

\numberwithin{equation}{section}

\maketitle

\setcounter{tocdepth}{1}
\numberwithin{equation}{section}

\addtocontents{toc}{\protect\setcounter{tocdepth}{1}}
\addtocontents{toc}{\protect\setcounter{tocdepth}{2}}

\section{Introduction}

\label{sec:intro}

Recently, the study of super-analogues of Teichm\"uller spaces and supermoduli spaces achieved much progress. In particular, Penner-type coordinates were discovered in \cite{penzeit}, \cite{ipz}, \cite{ipz2} for $N=1, N=2$ versions of Teichm\"uller space and their decorated analogues. In that particular case, these spaces were viewed as a subspace of the character variety $\Hom(\pi_1(F),G)/G$, where $\pi_1(F)$ 
is the fundamental group of a hyperbolic Riemann surface  $F$ with punctures. For the standard Teichm\"uller space $T(F)$, $G=PSL(2,\mathbb{R})$, while the super-analogues $ST_{N=1}(F)$, $ST_{N=2}(F)$ cases are related to rank $1$ and rank $2$ supergroups $OSP(1|2)$, $OSP(2|2)$ correspondingly. 

Penner coordinates \cite{DTT}, \cite{penner} are essential not only in the study of hyperbolic geometry, but they also provide a geometric example of a fundamental algebraic object, known as a cluster algebra. Let us briefly characterize the context. One starts from an ideal triangulation of the Riemann surface, or, equivalently, the dual trivalent fatgraph (aka ribbon graph). Penner coordinates assign a parameter for every edge of triangulation/fatgraph related to a suitably renormalized geodesic length. This provides coordinates for a trivial bundle $\tilde{T}(F)$ over $T(F)$, known as the decorated Teichm\"uller space. There is a simple transformation to coordinates on Teichm\"uller space so that these new coordinates are subject to linear constraints. One of the benefits of such coordinates is that the 
action of mapping class group on $\tilde{T}(F)$ is described in a combinatorial way by embedding into the Ptolemy groupoid. The Ptolemy groupoid is generated by elementary moves on the fatgraphs, called flips. The action of flips on Penner coordinates gives an example of the so-called cluster transformations. 
The other benefit of Penner coordinates is that they serve as  Darboux-type coordinates for the Weil-Petersson 2-form, which makes them useful for the quantization of $T(F)$ \cite{Kashaev97}, \cite{Chekhov99}.  
 
The supergroups $OSP(1|2)$, $OSP(2|2)$ which give rise to super-Teichm\"uller spaces both contain $G=SL(2,\mathbb{R})$ as their body subgroup. Penner coordinates have been generalized successfully in both of these cases, leading to the super-analogue of Ptolemy transformation. Currently, there are a lot of attempts to construct a super-analogue of cluster algebras based on these formulas \cite{ovshap},\cite{musikoven}, \cite{musikoven2}. The critical ingredient of both constructions were the graph $G$-connections: $\mathbb{Z}_2$-graph connections describing the spin structures on $F$ for $ST_{N=1}(F)$, and two such spin structures accompanied by $GL(1)_+$-graph connection for $ST_{N=2}(F)$. In particular, part of the decoration for $ST_{N=2}(F)$ was related to the gauge equivalences of $GL(1)_+$-graph connections. 
One of the choices of root systems for $OSP(2|2)$ is such that simple roots are ``grey", namely each of them gives rise to a $GL(1|1)$ subgroup. $GL(1|1)$ is a reductive supergroup of rank $1$, which contains two abelian subgroups as its body. Thus, only the odd coordinates are responsible for non-commutativity. 

In this note, we study the first nontrivial case of a character variety related to simple supergroups, namely
$${\rm M}_{GL(1|1)}=\Hom(\pi_1(F), GL(1|1))/GL(1|1),$$ 
which is quite interesting on its own. We will consider this space as the space of $GL(1|1)$-graph connections on the trivalent fatgraph associated with $F$. Then, we will define coordinates on this space through the assignment of specific parameters to edges of the fatgraph. These parameters are related to Gaussian decomposition of $GL(1|1)$. Using these coordinates, we obtain a characterization of the action of the Ptolemy groupoid. Finally, we discuss a Poisson bracket structure on ${\rm M}_{GL(1|1)}$.

We believe that $GL(1|1)$ character variety is essential in the context of the super-analogue of abelianization \cite{hn}. Another important context is the quantum $GL(1|1)$ Chern-Simons theory \cite{Rozansky92}, which recently attracted some attention \cite{Aghaei18}. 

The structure of the paper is as follows. Section 2 reviews some of the notions of super mathematics and necessary facts about the $GL(1|1)$ supergroup. Section 3 discusses fatgraphs, $G$-graph connections, and Ptolemy groupoid actions on trivalent fatgraph $G$-connections. Section 4 is devoted to the construction of a coordinate system on ${\rm M}_{GL(1|1)}$ and its decorated version using its fatgraph $GL(1|1)$-connection realization through Gaussian decomposition of $GL(1|1)$. Section 5 is primarily of a computation nature, where we describe the action of flip transformation using the minimal amount of changing variables in the decorated space. Finally, in Section 6, we discuss the Poisson bracket structure on ${\rm M}_{GL(1|1)}$.

\subsection*{Acknowledgements}  We thank R.C. Penner for useful discussions and his comments on the manuscript. A.M.Z. is partially supported by Simons Collaboration Grant 578501 and NSF grant DMS-2203823.

\section{$GL(1|1)$ supergroup}
\subsection{Conventions on superspaces and Grassmann algebras}

In this section we follow the conventions from \cite{Supermanifolds}.

Let $\R^{S[N]}$ the real Grassmann algebra with generators $1, \beta_{[i]}$, for $i=1,2,...,N$. The generators have relations $1\beta_{[i]}=\beta_{[i]}=\beta_{[i]} 1$ and $\beta_{[i]}\beta_{[j]}=-\beta_{[j]}\beta_{[i]}$. In particular, $(\beta_{[i]})^2=0$. We also use the notation $\beta_{[\lambda]} = \beta_{[\lambda_1]}\cdot\cdot\cdot\beta_{[\lambda_k]}$ for an ordered multi-index $\lambda =\lambda_1,...,\lambda_k$. Note that if $\lambda$ has a repeated index, then $\beta_{[\lambda]}=0$. If the multi-index is empty, the corresponding element is the empty product, 1. By our commutation relations, any $\beta_{[\lambda]}$ can have its terms rearranged so that the indices of the terms are increasing. Thus, an element in $\R^{S[N]}$ can be written in the form $x=\sum_\lambda x_\lambda \beta_{[\lambda]}$ for $x_\lambda \in \R$ as $\lambda$ runs over all strictly increasing multi-indices. 

\begin{Def}
The \textit{degree} of a term $x_\lambda\beta_{[\lambda]}\in \R^{S[N]}$ is defined as the size of the multi-index $\lambda$. 
\end{Def}

Thus $\R^{S[N]}$ has a superalgebra structure given by the decomposition $\R^{S[N]}=\R^{S[N]}_0\oplus \R^{S[N]}_1$ into elements which are sums of terms of even (respectively, odd) degree. Since the $\beta$ generators anti-commute, $\R^{S[N]}$ is supercommutative. 

\begin{Def}
The \textit{body map} $\epsilon:\R^{S[N]}\to \R$ is the projection of an element onto its coefficient of $1$. The \textit{soul map} $s:\R^{S[N]}\to \R^{S[N]}$ is defined by $s(x)=x-\epsilon(x)\cdot 1$.
\end{Def}

Since there are $N$ anti-commuting generators, we have that $s(x)^{N+1} = 0$. Thus $\epsilon(x)\neq 0$ if and only if $x$ is invertible. Explicitly, $$x^{-1}=\dfrac1{\epsilon(x)}\left(1-\dfrac{s(x)}{\epsilon(x)} + \left(\dfrac{s(x)}{\epsilon(x)}\right)^2-... + (-1)^N\left(\dfrac{s(x)}{\epsilon(x)}\right)^N \right).$$ In a similar vein, we can use series expansions to define, say, $\sqrt x$ for $x$ with positive body, as the series will terminate. \newline

\rmk Inequalities such as $x>0,x<0,x\neq 0$ for $x\in \R^{S[N]}$ will be taken to be inequalities on the body $\epsilon(x)$. \newline

\begin{Def}
Given $ \R^{S[N]} = \R^{S[N]}_0\oplus \R^{S[N]}_1$, the \textit{superspace} $\mathbb R^{1|1}$ is defined as $\R^{S[N]}_0\times \R^{S[N]}_1$. 
\end{Def}

In other words, we have a two-dimensional space with one even and one odd coordinate. We can define more generally $\R^{p|q}=\(\R^{S[N]}_0\)^p\times \(\R^{S[N]}_1\)^q$. \\

\rmk From now on, we will use the convention that odd elements are denoted by Greek letters, and even elements are denoted by Latin letters.

\subsection{$GL(1|1)$ supergroup and its Lie superalgebra}

\begin{Def}
A \textit{Lie superalgebra} is a superalgebra whose multiplication, denoted by $[X,Y]$, is super-anticommutative, and furthermore satisfies the super Jacobi identity $(-1)^{|X||Z|}[X,[Y,Z]] + (-1)^{|Z||Y|}[Z,[X,Y]] + (-1)^{|Y||X|}[Y,[Z,X]] = 0$.
\end{Def}

Let us introduce a Lie superalgebra $\mf{gl}(1|1)$. It has two even generators $E, N$ and two odd generators $\Psi^{\pm}$, which satisfy the following commutation relations: 
\begin{eqnarray}
[N,\Psi^\pm]=\pm\Psi^\pm,\quad [\Psi^+,\Psi^-]=E, \quad [E, \Psi^{\pm}]=[E,N]=0.
\end{eqnarray}
In the defining representation, as elements of ${\rm End}(\mathbb{R}^{1|1})$ these generators are given by the supermatrices below:
\begin{eqnarray}
E = \pmat 1001, N = \pmat {\frac12}00{-\frac12}, \Psi^+=\pmat 0100, \Psi^-=\pmat 0010.
\end{eqnarray}

\begin{Def}
$GL(1| 1)$ is the group of invertible linear transformations of $\mathbb R^{1| 1}$. Elements of $GL(1|1)$ can be identified as supermatrices of the form $\pmat a\alpha \beta b$ for $a,b\in \R^{S[N]}_0,\alpha,\beta\in\R^{S[N]}_1$ with $\epsilon(a),\epsilon(b)\neq 0$. 
\end{Def}

In what follows, $GL(1|1)$ will refer to the identity component of this group, which means the even entries $a,b$ will have positive body. 

The Lie superalgebra of $GL(1|1)$ is given by $\mathfrak{gl}(1|1)$, so that any element of $GL(1|1)$ can be represented as $e^{R}$, where  $R=nN+eE+\psi_{+}\Psi^{+}+\psi_{-}\Psi^{-}$.  

To keep the Lie superalgebra/Lie supergroup correspondence explicit, we choose the following multiplication of two elements in $GL(1| 1)$ as follows: 
\begin{eqnarray}
\pmat a\alpha\beta b \pmat c\gamma\delta d = \pmat{ac-\alpha\delta}{a\gamma+d\alpha}{c\beta+d\delta}{bd-\beta\gamma}.
\end{eqnarray}

In particular, we note the minus signs in the multiplication formula.  Let us elaborate on this choice. We identify $e^{\alpha\Psi^+}=1+\alpha\Psi^+$ $=\pmat1001 + \alpha\pmat0100$ with $\pmat1\alpha01$. Then, when multiplying $\pmat1\alpha01\pmat10\beta1$, the result should agree with $(1+\alpha\Psi^+)(1+\beta\Psi^-)$. Upon expanding the latter product, we have the term $\alpha\Psi^+\beta\Psi^-$. Since $\beta$ and $\Psi+$ are both odd, we have $\Psi^+\beta = -\beta\Psi^+$. Thus $$\pmat1\alpha01\pmat10\beta1 = 1 + \alpha\Psi^+ + \beta\Psi^- - \alpha\beta\Psi^+\Psi^-=\pmat{1-\alpha\beta}\alpha\beta1,$$ which agrees with our choice of multiplication.

Other references may define supermatrix multiplication without these extra signs. There is an isomorphism $\pmat a\alpha\beta b\mapsto\pmat a{-\alpha}\beta b$ from our convention to the other convention.

There is also a notion of supertrace, namely $str\pmat a\a\b b = a-b$. This gives rise to a nondegenerate invariant bilinear form on $\mf{gl}(1|1)$. This takes the role of the Killing form, which is degenerate in this case.

\subsection{Parametrization of $GL(1|1)$ and its decomposition}

Any element in $GL(1|1)$ admits the following unique Gaussian factorization: 
\begin{eqnarray}\label{gauss}
\pmat a\alpha\beta b = \pmat 1{\frac\alpha b}01\pmat{a+\frac{\alpha\beta}b}00b\pmat 10{\frac\beta b}1.
\end{eqnarray}
We choose the following parametrization of $g\in GL(1|1)$:
\begin{eqnarray}\label{paramg}
g(a,b; \alpha,\beta )=\param ab\alpha\beta.
\end{eqnarray}
The scalar factor in front is meant as an element propotional to the unit matrix as an element of $GL(1|1)$.
This presentation is unique and thus provides a parametrization of $GL(1|1)$. Given an element in the standard form, say $\pmat A\phi\psi B$, its coordinates in this parametrization are 
\begin{eqnarray} 
a=\sqrt{AB}, b = \sqrt{\dfrac BA}\left( 1-\dfrac{\phi\psi}{2AB} \right) ; \alpha = \dfrac\phi B, \beta = \dfrac\psi B.
\end{eqnarray}

The reason to introduce the extra factor in front is that it provides compact formula for the inverse element. We write the formulas for multiplication and inversion in these coordinates:
\begin{eqnarray}
&&(a,b;\alpha,\beta)(c,d;\gamma,\delta) = (ac(1-\frac12\b\gamma)(1-\frac12b^2d^2\alpha\delta),bd;\alpha+b^{-2}\gamma,d^{-2}\beta+\delta),\nonumber\\
&&(a,b;\alpha,\beta)^{-1} = (a^{-1},b^{-1};-b^2\alpha,-b^2\beta).
\end{eqnarray}

\section{Fatgraphs and graph connections}

We consider surfaces $F$ with genus $g\geq 0$ and $s\geq 1$ punctures, such that $2g-2+s>0$. 

\begin{Def}
A \textit{fatgraph} (also known as \textit{ribbon graph}) is a graph with a cyclic ordering of the edges at each vertex. An \textit{orientation} on a fatgraph is an assignment of direction to each edge of the graph.
\end{Def}

One can reconstruct $F$ from a fatgraph $\tau$ by fattening the edges and verticies along with cyclic ordering of edges at the vertices (see e.g. \cite{DTT} for details). The condition $2g-2+s>0$ is necessary to choose a trivalent fatgraph for a given surface, as the number of vertices of such a graph is equal to $2(2g-2+s)$.
There are Whitehead (flip) transformations as in Figure \ref{fig:graphflip}, which take the edge $e$ between 
vertices $u, v$, adjacent to $c,d$ and $a, b$ respectively, shrinks it, and then extends an edge $f$ connecting vertices $u',v'$ adjacent to $b,c$ and $a,d$ respectively. Flips are known to act transitively on the collection of trivalent fatgraphs of $F$. Altogether they form a {\it Ptolemy groupoid ${\rm Pt}(F)$}. 
Composition of flips are known to give generators for the mapping class group of a surface \cite{DTT}.

Given a fatgraph $\tau$ and a Lie (super)group $G$, we define a $G$-graph connection on $\tau$ as follows.

\begin{Def}
A $G$-\textit{graph connection} is the assignment to each edge $e$ of $\tau$ an element $g_{e}\in G$ and an orientation, so that $g_{\bar{e}}=g_e^{-1}$ if $\bar{e}$ is the inverse-oriented edge. We denote the space of $G$-graph connections as $\tilde{\rm M}_G(\tau)$. 
Two graph connections $\{g_e\}$, $\{\tilde{g}_e\}$ are \textit{gauge equivalent} if there is an assignment $v\rightarrow h_v\in G$, for all vertices $v\in \tau$, such that $\tilde{g}_e=h_v^{-1}g_eh_{v'}$, where $e$ starts and ends at $v'$ and $v$ respectively. We will denote the space of gauge equivalence classes of elements in $\tilde{\rm M}_G(\tau)$ by ${\rm M}_G(\tau)$. 
\end{Def}

The space of graph connections modulo gauge equivalences can be identified with a more common differential-geometric object. Namely, there is a natural one-to-one correspondence between ${\rm M}_G(\tau)$ and the moduli space of flat $G$-connections. This correspondence is constructed as follows: one identifies $M_A(e)$, the monodromy of the flat connection $A$ along the oriented edge of the fatgraph $e$, with the group element $g_e$. This way, the space $\tilde{\rm M}_G(\tau)$ is identified with the space of flat connections modulo gauge transformations which are equal to identity at the vertices. Altogether, compositions of these group elements along the cycles of the fatgraph contain all information about the gauge classes of $A$. However, there is a residual gauge symmetry at the vertices of the fatgraph, which one has to take into account, and that is precisely the equivalence relation for graph connections. 

One can formulate this as follows.
\begin{thm}
If $F$ deformation retracts to $\tau$, then the moduli space of flat $G$-connections on $F$ is isomorphic to the space of gauge equivalent classes of $G$-graph connections on $\tau$ corresponding to $F$, i.e. 
\begin{equation}
{\rm M}_G(\tau)\cong \Hom(\pi_1(F), G)/G 
\end{equation}
\end{thm}

For any element $A\in \Hom(\pi_1(F), G)/G $, let us denote $A(\tau)$ the corresponding graph connection on a fatgraph $\tau$. 





We are interested in understanding the action of ${\rm Pt}(F)$ on graph connections. 
We require that for any $\Gamma\in {\rm Pt}(F)$ and any cycle $c\in \tau$, if $A$ is the flat connection, the transformed connection $A^{\Gamma}$ is such that $M_{A^{\Gamma}}(\Gamma(c))=M_A(c)$. 
This gives a natural action of ${\rm Pt}(F)$ on $\Hom(\pi_1(F), G)/G $.
    
To write it formally, we look at the elementary flip transformation as in Figure \ref{fig:graphflip}, 
which involves 5 edges. 

\begin{figure}
    \centering
    \includegraphics[scale=0.4]{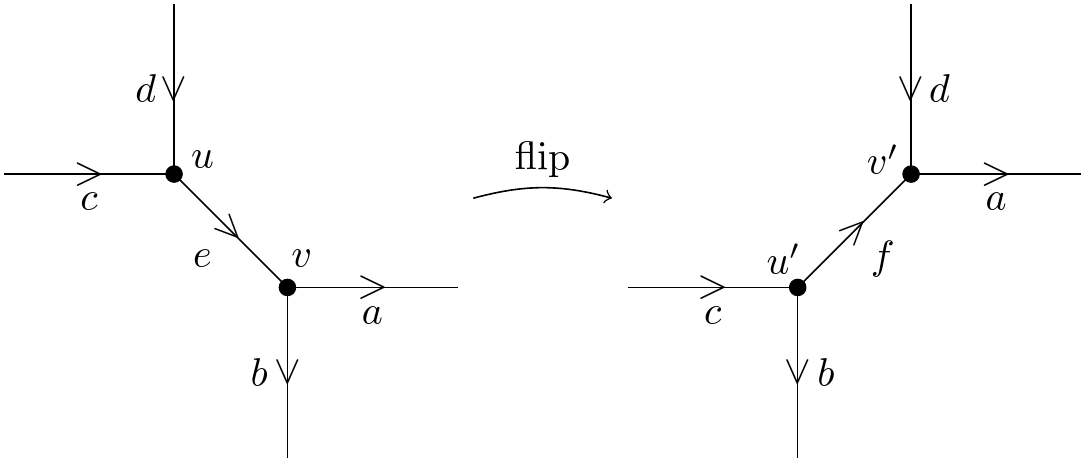}
    \caption{Whitehead flip of a fatgraph, at the edge $e$.}
    \label{fig:graphflip}
\end{figure}

There are 6 pieces of possible cycles on such a fatgraph, corresponding to a choice of 2 of the 4 boundary edges.  The monodromy conservation along these pieces after the flip transformation thus leads to the following equations:
\begin{align*}
    g_ag_eg_c &= g_a'g_fg_c', \\
    g_bg_eg_d &= g_b'g_f^{-1}g_d', \\
    g_ag_eg_d &= g_a'g_d', \\
    g_d^{-1}g_c &= {g_d'}^{-1}g_fg_c', \\
    g_bg_eg_c &= g_{b}'g_{c}', \\
    g_ag_b^{-1} &= g_a'g_f{g_b'}^{-1}.
\end{align*}

The solution to those equations is unique up to equivalence and given by the following theorem.

\begin{thm}
For any flip $\Gamma$ of $\tau$, 
the transformation $A\to A^{\Gamma}$ is such that 
the $G$-graph connection corresponding to $A^{\Gamma}(\Gamma(\tau))$ is related to $A(\tau)$ as depicted in Figure \ref{fig:graphflips}, where
\begin{figure}
    \centering
    \includegraphics[scale=0.4]{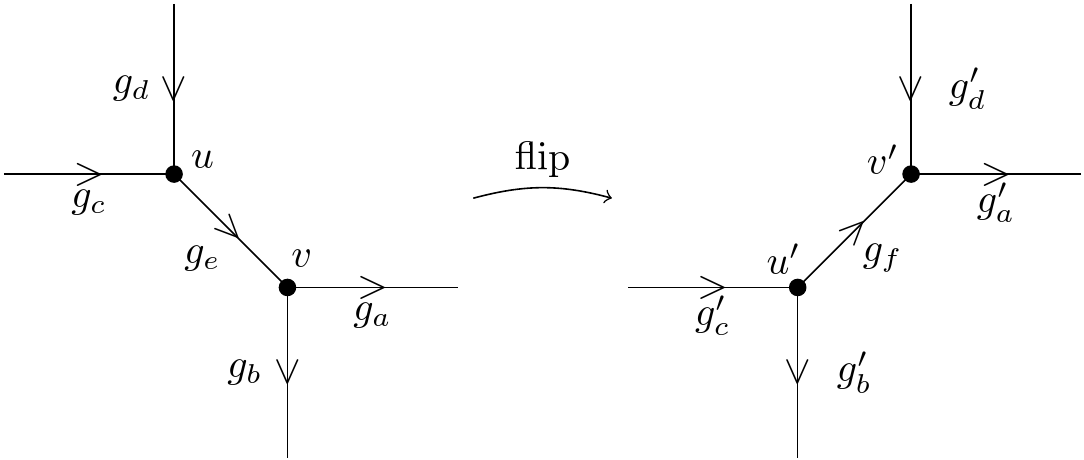}
    \caption{Flip transformation on a graph connection.}
    \label{fig:graphflips}
\end{figure}
\begin{align*}
    g_a' = g_a &\hspace{3em} g_{b}' = g_bg_e \\
    g_c' = g_c &\hspace{3em} g_d' = g_eg_d \\
    &g_f = g_e.
\end{align*}
\end{thm}


One can verify the relations explicitly by the choice of gauge transformation $h_{u'}=c'c^{-1}$ and $h_{v'}=(a')^{-1}a$. However, another method is the following trick. One can use gauge transformations to make $g_e=1$ and then shrink that edge, since it does not participate in monodromy. This results in a 4-valent vertex. Expanding the edge makes two trivalent vertices and the edge $f$. After that, one makes a gauge transformation to achieve
$g_f=g_e$.

We note that this solution works for any group; however, it is asymmetric. Group elements spread in one direction of the graph but not the other. In the next section we will put the coordinates on ${\rm M}_G(\tau)$,
represented as graph connections in the case of $G=GL(1|1)$. We will find the formulas for the flip transformations in those coordinates which remove this spreading effect.

\section{Coordinates on the moduli space of $GL(1|1)$-graph connections}

We fix a surface $F$ with genus $g\geq 0$ and $s\geq 1$ punctures such that $2g+s-2>0$. We also fix a trivalent fatgraph $\tau\subset F$ with orientation $o$.

Assigning to every edge a group element of $GL(1|1)$, or alternatively assigning an ordered tuple $(a,b;\alpha,\beta)$ in the parametrization (\ref{paramg}), we obtain a vector in the coordinate system $\tilde{C}(F,o,\tau)$ for the space of $GL(1|1)$-graph connections without factoring by the gauge group equivalences.
Thus the chart $\tilde{C}(F,o,\tau)$ gives a diffeomeorphism:
\begin{eqnarray}
\tilde{\rm M}_{GL(1|1)}(\tau)\cong \mathbb{R}^{12g-12+6s|12g-12+6s}_+,
\end{eqnarray}

However, we need to take into account that one can change the orientation on edges of the graph, thus leading to an equivalent coordinate system on $\tilde{\rm M}_G(\tau)$.

\begin{Def}
We say that two charts $\tilde{C}(F,o,\tau)$ and $\tilde{C}(F,\bar{o},\tau)$, corresponding to orientations $o$, $\bar{o}$, are \textit{equivalent} if the assignments agree on the edges where orientation is the same, but are related by the formula $(\bar{a},\bar{b};\bar{\alpha},\bar{\beta})=(a^{-1},b^{-1};-b^2\alpha,-b^2\beta)$, where the orientation is reversed.
\end{Def}



Our goal is now to reformulate the gauge transformations at the vertices of a fatgraph as relations between coordinates in the charts $\tilde{C}(F,o,\tau)$. 

\begin{Def}
Let $\vec c\in \tilde C(F,o,\tau)$ be a coordinate vector. Suppose $e_1,e_2,e_3$ are edges oriented towards a vertex $u$ such that, for each $i=1,2,3$, $e_i$ has coordinates $(a_i,b_i;\alpha_i,\beta_i)$. Then a \textit{vertex rescaling} at $u$ is one of three coordinate changes for the coordinates of each $e_i$, where $c$ is positive and even and $\gamma$ is odd:

\begin{itemize}
    \item $(a_i,b_i;\alpha_i,\beta_i)\mapsto (a_ic^{-1},b_i;\alpha_i,\beta_i)$;
    \item $(a_i,b_i;\alpha_i,\beta_i)\mapsto (a_i,b_ic^{-1};c^2\alpha_i,\beta_i)$;
    \item $(a_i,b_i;\alpha_i,\beta_i)\mapsto (a_i(1-\frac12b_i^2\b_i\gamma),b_i;\alpha_i-\gamma,\beta_i)$.
\end{itemize}

If $e_1,e_2,e_3$ are oriented away from a vertex $u$, then there is also a vertex rescaling at $u$ for odd $\gamma$: 

\begin{itemize}
    \item $(a_i,b_i;\alpha_i,\beta_i)\mapsto (a_i(1-\frac12b_i^2\alpha_i\gamma),b_i;\alpha_i,\beta_i+\gamma)$.
\end{itemize}
\end{Def}

It turns out that the edge reversal and the vertex rescalings define an equivalence relation on the chart $\tilde{C}(F, o, \tau)$. The vertex rescalings come from the equivalences on $GL(1|1)$ graph connections provided by the appropriate 1-parameter subgroups. This leads to the following theorem.

\begin{thm}
Let $C(F,\tau) = \tilde C(F,o,\tau)/\sim$ be the quotient by the equivalences provided by edge reversal and vertex rescalings. Then $C(F,\tau)$ is in bijection with the moduli space ${\rm M}_{GL(1|1)}(\tau)$.
\end{thm}

\begin{proof}
Indeed, explicitly, the coordinates $(a,b;\alpha,\beta)$ correspond to \begin{align*}
    \param ab\alpha\beta
\end{align*}
in $GL(1|1)$. The inverse of such an element is 
\begin{align*}
    \parami ab\alpha\beta,
\end{align*}
which corresponds to the coordinates $(a^{-1},b^{-1};-b^2\alpha,-b^2\beta)$. Therefore, the rule for reversing orientations in $\tilde{C}(F,o,\tau)$ matches the rule for reversing orientations of a graph connection on $\tau$. 

Now consider a graph connection on $\tau$. For an edge $e$ directed towards a vertex $u$ in $\tau$, there is a gauge transformation $g_e\mapsto h_u^{-1}g_e$, where $g_e$ is a group element on $e$, and $h_u$ is a group element associated to the vertex $u$. Similarly, if $e$ is directed away from $u$, the gauge transformation is $g_e\mapsto g_eh_u$. This gauge transformation at $u$ acts on each edge adjacent to $u$. In the case of a trivalent $\tau$, there are three edges $e_1,e_2,e_3$ at a given vertex $u$. Then the vertex rescalings correspond to gauge transformations at $u$, where $h_u$ is one of the following:

\begin{itemize}
    \item $h_u = c$,
    \item $h_u = \pmat{c^{-1}}00c$,
    \item $h_u = \pmat1\gamma01$,
    \item $h_u = \pmat10\gamma1$.
\end{itemize}
Again, $c$ is positive even and $\gamma$ is odd. In the first three gauge transformations, we require that $e_1,e_2,e_3$ all point towards $u$, and for the fourth, we require that the edges all point away from $u$. The claim that the vertex rescalings correspond to gauge transformations follows from routine multiplication in $GL(1|1)$. For instance, in the third scenario, we have \begin{align*}
    &g(1,1;\gamma,0)^{-1}g(a,b;\a,\b) = g(a(1-\frac12b^2\beta\gamma),b;\a+\gamma,\b).
\end{align*}

These four gauge elements generate all possible gauge transformations, since each of the four elements together generate $GL(1|1)$. 
\end{proof}

It is not hard to see that one can simply constrain the vertex rescalings. To do that let us introduce the following notation: for a given edge with assignment $(a,b;\alpha, \beta)$, which is adjacent to vertex $v$, we denote $a^v=a$, $b^v=b$, $\alpha^v=\alpha$, $\beta^v=\beta$, if edge is oriented towards $v$, and $a^v=a^{-1}$ $b^v=b^{-1}$, $\alpha^v=-\alpha$, $\beta^v=-\beta$ otherwise.
Then, let us give the following definition.
\begin{Def}
Let $\vec{c}\in\tilde{C}(F,o,\tau)$ so that $\{(a_i, b_i; \alpha_i, \beta_i)\}_{i=1,2,3}$ are the assignments for edges adjacent to vertex $v$. We refer to the following conditions as \textit{gauge constraints} at vertex $v$:
\begin{eqnarray}
\prod^3_{i=1}{a}^v_i=\prod^3_{i=1}{b}^v_i=1,\quad 
\sum_{i=1}^3{\alpha^v_i}=\sum_{i=1}^3{\beta^v_i}=0
\end{eqnarray}
\end{Def}

One can see that, for fixed orientation $o$, such constraints pick exactly one element from the equivalence classes provided by vertex rescalings. Moreover, we have the following theorem.

\begin{thm}
The coordinate chart $\tilde{C}(F,o,\tau)$ modulo gauge constraints at all vertices $v$ of $\tau$ provides a real-analytic isomorphism
\begin{eqnarray}
{\rm M}_{GL(1|1)}(\tau)\cong\mathbb{R}_+^{4g+2s-2|4g+2s-2},
\end{eqnarray}
so that $\tilde{\rm M}_{GL(1|1)}(\tau)$ is a trivial bundle over $M_{GL(1|1)}(\tau)$ with the fiber $\mathbb{R}_+^{8g-10+4s|8g-10+4s}$.
\end{thm}

There are benefits of working with the ``decorated space" $\tilde{\rm M}_{GL(1|1)}(\tau)$ instead of ${\rm M}_{GL(1|1)}(F)$. As we shall see in the next section, the gauge freedom allows us to write a compact formula for a flip transformation in our coordinates, somewhat in the spirit of Penner coordinates on decorated Teichm\"uller space.

\section{Minimal formula for the flip transformations}

First we describe the suitable description of the elements of ${M}_{GL(1|1)}(F)$ in terms of coordinates in the extended chart $\tilde{C}(F,o,\tau)$.

Namely, we assign coordinates $(a,b;\alpha,\beta)$, corresponding to a $GL(1|1)$ group element, to an oriented edge as follows: the odd parameters $\alpha, \beta$ are assigned to the terminal and initial points of the edge, while we put the pair of even parameters $(a,b)$ between them, as depicted in Figure
\ref{fig:coordorient}. This ordering goes along with the Gaussian decomposition from Section 2.

\begin{figure}
    \centering
    \includegraphics[scale=0.4]{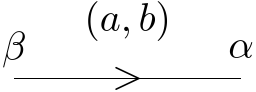}
    \caption{Coordinates on an edge.}
    \label{fig:coordorient}
\end{figure}

Let us now describe the minimal representation of flip transformation which involves change of the minimal number of parameters. We will fix the odd elements at the ends, as well as eliminate the dependence on fermionic coordinates in the mid-edge after the flip, as depicted in Figure \ref{fig:fixendsol}.


\begin{figure}
    \centering
    \includegraphics[scale=0.4]{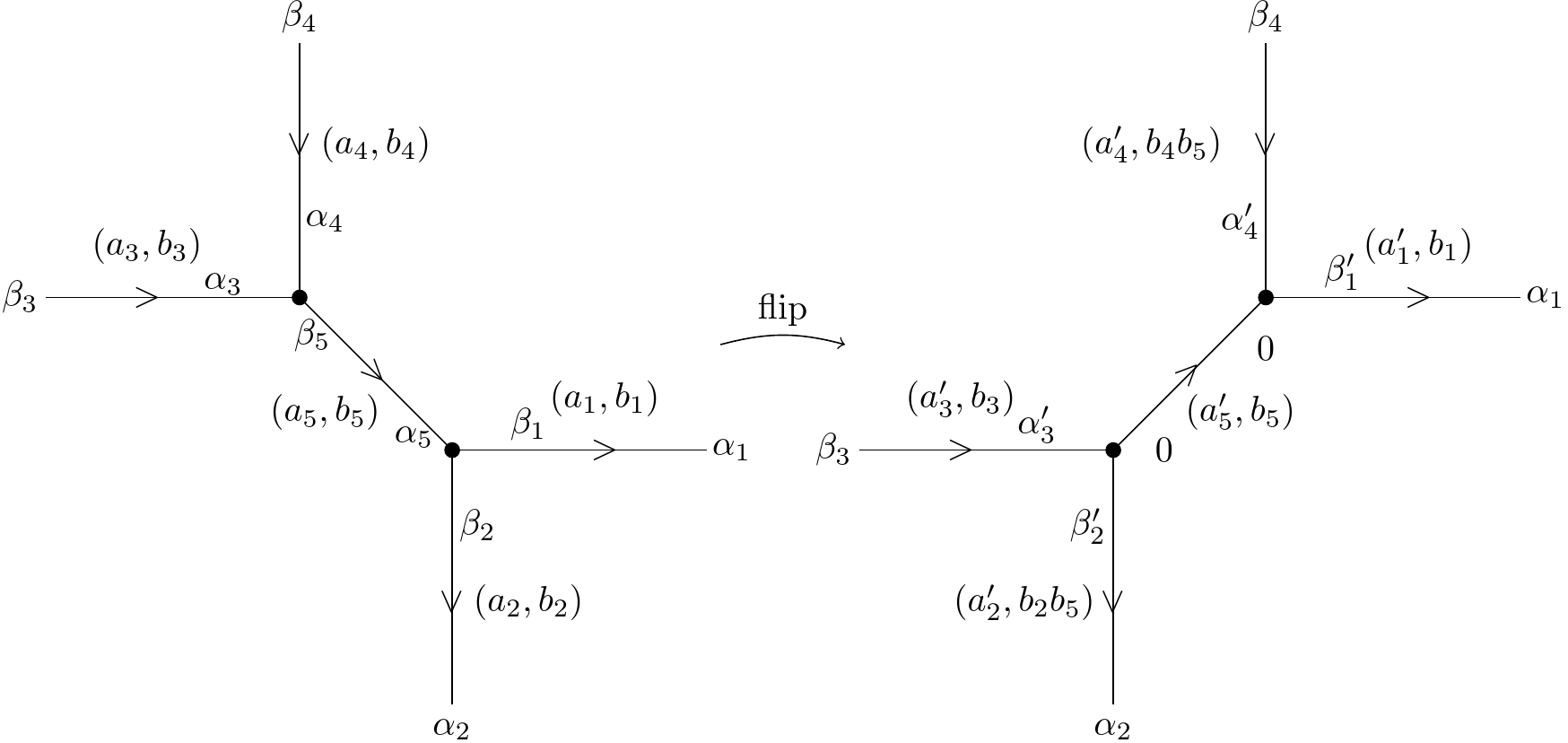}
\caption{Minimal form of the flip transformation}
    \label{fig:fixendsol}
\end{figure}

\begin{thm}
Let $\tau'$ represent the flip of $\tau$ at an edge $e$, with orientations as depicted in Figure \ref{fig:graphflip}. 
Let $\vec c\in \tilde{C}(F,o,\tau)$ and let $\vec {c'}$ be the corresponding element in 
$\tilde{C}(F,o',\tau')$ modulo gauge equivalences. 
Choose coordinates for $\vec c$ such that $\{(a_i,b_i;\alpha_i,\beta_i)\}_{i=1,\dots,5}$ represent the edges where the flip occurs. 
Then there is a choice for $\vec {c'}$, which is described on the Figure \ref{fig:fixendsol}, where the primed variables are given by the following formulas:
\begin{eqnarray}
&&\beta_1' = \beta_1 + b_5^2\beta_5,~ \beta_2'=\beta_5+b_5^{-2}\beta_2,~ \alpha_3' = \alpha_3+b_5^2\alpha_5, \alpha_4'=b_5^{-2}\alpha_4+\alpha_5\nonumber\\ 
&&a_1'=a_1(1-\frac12b_1^2b_5^2\alpha_1\beta_5),~ a_3'=a_3(1-\frac12b_3^2b_5^2\alpha_5\beta_3),~ a_5'=a_5(1-\frac12b_5^2\alpha_5\beta_5),\nonumber\\ 
&&a_2'=a_2a_5(1-\frac12b_2^2b_5^2\alpha_2\beta_5)(1-\frac12b_5^2\alpha_5\beta_5),~  
    a_4'=a_4a_5(1-\frac12b_4^2b_5^2\alpha_5\beta_4)(1-\frac12b_5^2\alpha_5\beta_5).   
\end{eqnarray}
\end{thm}

\begin{proof}
Recall that the product of elements $(a,b;\alpha,\beta)$ and $(c,d;\gamma,\delta)$ is $(ac(1-\frac12\b\gamma)(1-\frac12b^2d^2\alpha\delta),bd;\alpha+b^{-2}\gamma,d^{-2}\beta+\delta)$. 

Applying this formula to the general solution of the flip (Theorem 2) means that we can first write $\vec {c'}$ with coordinates 

\begin{align*}
    &(a_1,b_1;\alpha_1,\beta_1), \\ &(a_2a_5(1+\frac12\alpha_5\beta_2)(1-\frac12b_2^2b_5^2\alpha_2\beta_5),b_2b_5;\alpha_2+b_2^{-2}\alpha_5,b_5^{-2}\beta_2+\beta_5), \\ &(a_3,b_3;\alpha_3,\beta_3), \\ &(a_4a_5(1+\frac12\alpha_4\beta_5)(1-\frac12b_4^2b_5^2\alpha_5\beta_4),b_4b_5;\alpha_5+b_5^{-2}\alpha_4,b_4^{-2}\beta_5+\beta_4), \\ &(a_5,b_5;\alpha_5,\beta_5).
\end{align*}

We see that $\alpha_1$ and $\beta_3$ are already in place, so we do not want to change these. We reverse the orientations of the second, fourth, and fifth edges so that we may alter the $\alpha$ term of the second and the $\beta$ term of the fourth. This gives 

\begin{align*}
    &(a_2^{-1}a_5^{-1}(1-\frac12\alpha_5\beta_2)(1+\frac12b_2^2b_5^2\alpha_2\beta_5),b_2^{-1}b_5^{-1};-b_2^2b_5^2\alpha_2-b_5^2\alpha_5,-b_2^2\beta_2-b_2^2b_5^2\beta_5), \\ &(a_4^{-1}a_5^{-1}(1-\frac12\alpha_4\beta_5)(1+\frac12b_4^2b_5^2\alpha_5\beta_4),b_4^{-1}b_5^{-1};-b_4^2\alpha_4-b_4^2b_5^2\alpha_5,-b_4^2b_5^2\beta_4-b_5^2\beta_5), \\ &(a_5^{-1},b_5^{-1};-b_5^2\alpha_5,-b_5^2\beta_5).
\end{align*}

Now that the second, third, and fifth edges are pointed towards a vertex, we apply a move which adds $b_5^2\alpha_5$ to the first fermionic coordinate of these edges. 

\begin{align*}
    &(a_2^{-1}a_5^{-1}(1+\frac12b_2^2b_5^2\alpha_2\beta_5)(1+\frac12b_5^2\alpha_5\beta_5)),b_2^{-1}b_5^{-1};-b_2^2b_5^2\alpha_2,-b_2^2\beta_2-b_2^2b_5^2\beta_5), \\ &(a_3(1-\frac12b_3^2b_5^2\alpha_5\beta_3),b_3;\alpha_3+b_5^2\alpha_5,\beta_3), \\  &(a_5^{-1}(1+\frac12b_5^2\alpha_5\beta_5),b_5^{-1};0,-b_5^2\beta_5).
\end{align*}

Now, we can invert the second edge again to get it to its original orientation; its formula is saved for the end of the proof.

The first, fourth, and fifth edges are all pointed away from a common vertex. We can then apply a move that adds $b_5^2\beta_5$ to the second fermionic coordinate of these edges. This gives 

\begin{align*}
    &(a_1(1-\frac12b_1^2b_5^2\alpha_1\beta_5),b_1;\alpha_1,\beta_1+b_5^2\beta_5), \\ &(a_4^{-1}a_5^{-1}(1+\frac12b_4^2b_5^2\alpha_5\beta_4)(1+\frac12b_5^2\alpha_5\beta_5),b_4^{-1}b_5^{-1};-b_4^2\alpha_4-b_4^2b_5^2\alpha_5,-b_4^2b_5^2\beta_4), \\ &(a_5^{-1}(1+\frac12b_5^2\alpha_5\beta_5),b_5^{-1};0,0).
\end{align*}

We invert the fourth and fifth edges to match the original choice of orientation, giving the final coordinates for $\vec {c'}$:

\begin{align*}
    &(a_1(1-\frac12b_1^2b_5^2\alpha_1\beta_5),b_1;\alpha_1,\beta_1+b_5^2\beta_5), \\ &(a_2a_5(1-\frac12b_2^2b_5^2\alpha_2\beta_5)(1-\frac12b_5^2\alpha_5\beta_5),b_2b_5;\alpha_2,\beta_5+b_5^{-2}\beta_2), \\ &(a_3(1-\frac12b_3^2b_5^2\alpha_5\beta_3),b_3;\alpha_3+b_5^2\alpha_5,\beta_3), \\ &(a_4a_5(1-\frac12b_4^2b_5^2\alpha_5\beta_4)(1-\frac12b_5^2\alpha_5\beta_5),b_4b_5;b_5^{-2}\alpha_4+\alpha_5,\beta_4), \\ &(a_5(1-\frac12b_5^2\alpha_5\beta_5),b_5;0,0).
\end{align*}

These coordinates have the desired fixed fermionic ends.
\end{proof}

\section{Poisson bracket}

In this section, we discuss in detail the Poisson structure that can be endowed on the space of graph connections $\tilde{\rm M}_G(\tau)$. It descends onto ${\rm M}_G(\tau)$, so that the isomorphism with the space of flat connections, modulo the action of the gauge group, is a Poisson manifold isomorphism. We recall this construction from \cite{fock1998poisson}.

As before, let us fix a surface $F$ with genus $g\geq 0$ and $s\geq 1$ punctures, such that $2g+s-2>0$, and a trivalent fatgraph $\tau\subset F$. $\tilde{\rm M}_G(\tau)$ is the space of graph connections on $\tau$. Let $G^{\tau}$ be the gauge group; that is, a direct product of the group $G$ with itself, one copy for each vertex in $\tau$. 

We have seen that $G^{\tau}$ acts on $\tilde{\rm M}_G(\tau)$, which gives the equivalence of graph connections. In order for the action to be a Poisson map, $ G^\tau$ needs a Poisson-Lie structure. Then $G$ itself needs a Poisson-Lie structure. This can be attained by choosing an $r$-matrix, which involves the Lie algebra $\mf g$ of $G$ \cite{Chari:1994pz}. This is an object satisfying the so-called classical Yang-Baxter equation (CYBE), namely:

\begin{eqnarray} 
[r^{12},r^{13}]+[r^{12}, r^{23}]+[r^{13}, r^{23}]=0
\end{eqnarray}

We also assume that the symmetric part of the $r$-matrix is the Casimir element $\Omega$, i.e. 
\begin{eqnarray}\label{rsymm}
\frac{1}{2}(r^{12}+r^{21})=\Omega= \sum e_i\otimes e_i,
\end{eqnarray}
 based on the invariant bilinear form on $\mathfrak{g}$ (so 
 that $\{e_i\}_{i=1,\dots\rm{dim}\mathfrak{g}}$ is an orthonormal basis).  For example, consider simple Lie algebras with the standard triangular decomposition $\mathfrak{g}=\mathfrak{n}_+\oplus \mathfrak{h}\oplus \mathfrak{n}_-$. Here, $\mathfrak{h}$ is 
the Cartan part and $\mf n_{\pm}$ are spanned by positive and negative root elements. 
The basic classical $r$-matrix of $\mf g$ is given by 
$r=\Omega_{\mathfrak{h}}+2\sum_{\alpha >0}e_{\alpha}\otimes e_{-\alpha}$, 
where $\Omega_{\mathfrak{h}}$ is the Casimir element restricted to $\mathfrak h$, and $(e_{\alpha}, e_{-\alpha})=1$ for all $\alpha$ with respect to an invariant form. 
For $\mathfrak{gl}(1|1)$, one can construct a solution of CYBE in a similar fashion: the only difference is that the nondegenerate bilinear form is defined by supertrace in the defining representation, not by the Killing form, which is degenerate. This way, the solution to CYBE with the condition 
(\ref{rsymm}) is given by (see also \cite{Gizem04}):
\begin{eqnarray} 
&&r=E\otimes N+N\otimes E-2 \Psi^{+}\otimes\Psi^{-},\nonumber\\
&&\Omega=E\otimes N+N\otimes E- \Psi^{+}\otimes\Psi^{-}+\Psi^{-}\otimes \Psi^{+}
\end{eqnarray}
This natually provides a Poisson-Lie structure on 
$GL(1|1)$. 

Let us describe the construction in \cite{fock1998poisson} of a Poisson bracket on $\tilde{\rm M}_G(\tau)$, adapted to our needs (simple supergroups). Let us adopt the following notation for graph connections. To a half-edge adjacent to a given vertex $v$, we associate an element $g_v$. To the opposite half edge we associate the element $g_v^{-1}$. In this way, we do not need to refer to a specific orientation of the edge. We then have a map
\begin{eqnarray}\label{doubmap}
G\to G\times G, g\mapsto (g,g^{-1})
\end{eqnarray}
 associated to every edge of $\tau$, giving
a map $\tilde{{\rm M}}_G(\tau)\to G^{2(6g-6+3s)}$. 

Given an orthonormal basis $\{e_i\}_{i=1,\dots\rm{dim}\mathfrak{g}}$, one can define a tangent vector at $X_i(g) = L_i(g) - R_i(g^{-1})$ on the image of map (\ref{doubmap}), where $L_i, R_i$ are left- and right-invariant vector fields associated to $e_i$.  

The final ingredient needed to write the Poisson structure is to fix a linear order at each vertex of $\tau$, not just a cyclic one. This is called a ciliation in \cite{fock1998poisson}. Since $\tau$ is trivalent, this means each vertex has edges labelled $1,2,3$ in a cyclic manner. For a vertex of $\tau$, we write the labels of the edges as elements of the vertex.

This allows us to write the following Poisson bivector which is a sum of bivectors over the set of vertices $V(\tau)$:
\begin{eqnarray}
P=\sum_{v\in V(\tau)}P_v, \quad P_v=\sum_{\alpha<\beta}\sum_{i,j}r^{ij}(v)X_i^\alpha \wedge X_j^\beta +\frac12 \sum_{\alpha}\sum_{i,j}r^{ij}(v)X_i^\alpha\wedge X_j^\alpha 
\end{eqnarray}
where $\alpha,\beta$ are edges adjacent to $v$ ordered according to ciliation, and $r(v)$ is an $r$-matrix associated to vertex $v$.

Then one has the following theorem.

\begin{thm}\cite{fock1998poisson}
The bivector $P$ defines a Poisson structure on $\tilde{{\rm M}}_G(\tau)$, and the group $G^{\tau}$ acts on 
$\tilde{{\rm M}}_G(\tau)$ in a Poisson way.
\end{thm}

If each $r(v)$ has the same symmetric part, then denoting the antisymmetric part as $r_a(v)=\frac12(r(v)^{12}-r(v)^{21})$, we obtain:

\begin{eqnarray}
&&P_v=P_{v,h}+P_{v,t},\nonumber\\
&&P_{v,h}=\sum_{i,j}r_a^{ij}(v)X^v_i\otimes X^v_j, \quad P_{v,t}=\sum_{\alpha, \beta;i}(v, \alpha,\beta;i)X_i^{\alpha}\otimes X_i^{\beta}
\end{eqnarray}

Here $X^v_i=\sum_{\alpha}X_i^{\alpha}$ is the generator of gauge transformations at $v$ and is tangent to the orbits of $G^{\tau}$, and 
\begin{equation}
(v, \alpha, \beta;i)=
\begin{cases}
1\qquad\qquad \quad\alpha>\beta\\
0\qquad \qquad\quad \alpha=\beta\\
-(-1)^{p(i)} \qquad\alpha<\beta,
\end{cases}
\label{eq:vabi}
\end{equation}
where $p(i)$ stands for the parity of generator $e_i$. From here one can see that the choice of antisymmetric part of the $r$-matrix is irrelevant when considering the Poisson bracket on the quotient ${\rm M}_G(\tau)$.

In our case $r_a=-\Psi_+\otimes\Psi_--\Psi_-\otimes\Psi_+$. In $P_{v,t}$, we sum over the $i$ index to get terms corresponding to the symmetric part (Casimir element). Once we sum over $i$, there is no need for \ref{eq:vabi}, so we use the simplified $(v,\a,\b) = \textrm{sign}(\a-\b)$. Then we have

\begin{eqnarray}
&&P_{v,h}^{GL(1|1)}=-X^{v}_{\Psi_+}\otimes X^v_{\Psi_-}-X^{v}_{\Psi_-}\otimes X^v_{\Psi_+},\nonumber\\
&&P^{GL(1|1)}_{v,t}=\sum_{\alpha, \beta;i}(v, \alpha,\beta)(X_E^{\alpha}\otimes X_N^{\beta}+
X_E^{\alpha}\otimes X_N^{\beta}-X_{\Psi^+}^{\alpha}\otimes X_{\Psi^-}^{\beta}+X_{\Psi^-}^{\alpha}\otimes X_{\Psi^+}^{\beta}).
\end{eqnarray}

Given our system of coordinates $\tilde{C}(F, o, \tau)$ for $\tilde{{\rm M}}_G(\tau)$, one can easily write down the presentation of the generators $X^{\alpha}_i$. Let us assume that an edge $\mu$, with coordinates $(a,b;\alpha,\beta)$, is oriented towards a vertex $v$. Then
\begin{eqnarray}
&&X^{\mu}_E=a\partial_a, \quad X^{\mu}_N=-\frac{1}{2}b\partial_b+\alpha\partial_{\alpha}, \nonumber\\
&&X^{\mu}_{\Psi^{+}}=-\frac{1}{2}ab^2\beta \partial_a+\partial_{\alpha}, \quad X^{\mu}_{\Psi^{-}}=
-\frac{1}{2}\alpha a\partial_a+b^{-2}\partial_{\beta}.
\end{eqnarray}
We introduce a change of coordinates by rearranging the one-parameter subgroups, as follows:  
\begin{eqnarray}
&&\param ab\a\b = \hat a\pmat 1{\hat\a}01 \pmat 10{\hat\b}1\pmat{\hat{b}^{-1}}00{\hat b},
\nonumber
\end{eqnarray}
 
\begin{eqnarray}
&&g(a,b;\alpha,\beta) =g(\hat{a}(1-\frac{1}{2}\hat{\alpha}\hat{\beta}),\hat{b}; \hat{\alpha}, \hat{b}^{-2}\beta).
\end{eqnarray}
This change of coordinates simplifies the vector fields:
\begin{eqnarray}
&&X^{\mu}_E=\hat a\partial_{\hat a}, \quad X^{\mu}_N=-\frac{1}{2}\hat b\partial_{\hat b}+\hat\alpha\partial_{\hat\alpha}-\hat\beta\partial_{\hat\beta},
\nonumber\\
&&X^{\mu}_{\Psi^{+}}=\partial_{\hat\alpha}, \quad X^{\mu}_{\Psi^{-}}=
-\hat\alpha \hat a\partial_{\hat a}+\partial_{\hat \beta}.
\end{eqnarray}
These coordinates, given that all three edges are pointed at the vertex $v$, form a Darboux-like chart for the bivector $P_v$.

\bibliography{biblio}
\end{document}